\documentclass[a4paper,11pt]{amsart}
\usepackage{amsfonts,amscd,amsmath,amssymb,tikz}
\usepackage{pifont}
\usepackage{youngtab}
\usepackage{ytableau}

\newtheorem{thm}{Theorem}[section]
\newtheorem{lem}[thm]{Lemma}

\newtheorem{prop}[thm]{Proposition}
\newtheorem{define}[thm]{Definition}

\newtheorem*{teoA}{Theorem A}

\theoremstyle{definition}

\theoremstyle{remark}
\newtheorem{rem}[thm]{Remark}
\newtheorem{example}[thm]{Example}

\begin{document}

\title{Part bounds for the Sylow permutation characters of $S_n$}

\author{Lorenzo Vanzi}
\address{Dipartimento di Matematica e Informatica Ulisse Dini, Universita degli Studi di Firenze}
\email{lorenzo.vanzi1@edu.unifi.it}

\keywords{}


\begin{abstract}
We study the Sylow permutation character of the symmetric group at the prime $2$ and prove some new bounds on the number of parts of partitions corresponding to its constituents. 
\end{abstract}

\maketitle


\section{Introduction}
Let $G$ be a finite group, let $p$ be a prime and let $P$ be a Sylow $p$-subgroup of $G$.
The study of the interplay between the representation theory of $G$ and that of its Sylow subgroups has been an active research theme for decades \cite{N10}. 
The structure of the Sylow permutation character $(1_P)\uparrow^{G}$ encodes a great deal of information about the algebraic structure of the group $G$. 
For instance, the normality of $P$ in $G$ can be read off the irreducible constituents of $(1_P)\uparrow^{G}$ \cite{MN12, GLLV}.
Despite this, very little is known about the decomposition of $(1_P)\uparrow^{G}$ in general. 
In this article we focus our attention on the case where $G=S_n$ is a symmetric group. From now on we let $\Omega_p(n)$ be the subset of $\mathrm{Irr}(S_n)$ consisting of all the irreducible constituents of $(1_P)\uparrow^{S_n}$. 
Since irreducible characters of $S_n$ are naturally labelled by partitions of $n$, it is convenient to think of $\Omega_p(n)$ as a subset of $\mathcal{P}(n)$, the set of partitions of $n$.

In \cite{GL1} Giannelli and Law gave a complete description of $\Omega_p(n)$ for every odd prime $p$. In subsequent work, again considering only odd primes, they extended this analysis and described the irreducible constituents of $\theta\uparrow^{S_n}$, for every $\theta\in\mathrm{Irr}(P)$ \cite{GL3}. It is surprising that, despite the amount of information accumulated for odd primes, very little is known for the prime $p=2$. 

The aim of this note is to start the investigation of the set $\Omega_2(n)$. 
In order to present our main result, we briefly introduce some notation. Let $\mathcal{H}(n)$ be the set of \textit{hook partitions} of $n$. Namely, partitions of the form $(n-x, 1^x)$, for some $0\leq x\leq n-1$. 
An easy consequence of \cite{G17} is the complete description of the set $\Omega_2(n)\cap\mathcal{H}(n)$. For this reason, here we focus on \textit{non-hook partitions}. With this in mind, the question we ask ourselves is the following one: what is the maximal number $\ell\in\mathbb{N}$ such that every non-hook partition with at most $\ell$ parts lies in $\Omega_2(n)$?
We call $\Delta(n,k)$ the subset of $\mathcal{P}(n)\smallsetminus\mathcal{H}(n)$ consisting of partitions with at most $k$ parts. The main result of this paper is Theorem \ref{main} below, where for every $n$ we determine the maximal $\ell$ such that $\Delta(n,\ell)\subseteq \Omega_2(n)$. Ignoring a few exceptions for small values of $n$, our main result can be stated as follows.

\begin{teoA}
Let $n\in\mathbb{N}$ and let us assume that $n\geq 11$. Let $n=2^{a_1}+2^{a_2}+\cdots+2^{a_t}$ be its binary expansion, where $a_1>a_2>\cdots>a_t\geq 0$. Then $$\Delta(n,k)\subseteq\Omega_2(n)\ \text{if and only if}\ k\leq a_1+t-1.$$ 
\end{teoA}

Theorem A shows that the irreducible character corresponding to any non-hook partition with \textit{few} parts is an irreducible constituent of the $2$-Sylow permutation character. This result extends the analysis of the restriction to Sylow $2$-subgroups previously performed in \cite{G17} and \cite{GV24}. We conclude by mentioning that we also prove a complementary result to Theorem A. More precisely, in Proposition \ref{big} we show that the maximum number of parts of a partition in $\Omega_2(n)$ is $\lceil\frac{n}{2}\rceil$. This result may be already known to experts, but we were not able to find a proof of it in the literature.

\subsection*{Acknowledgments}
The author thanks Eugenio Giannelli for suggesting this research project and for his invaluable help with the manuscript.

\section{Background and Notation}\label{sec: 2}
In this section we will establish notation and present some preliminary results required for the proofs in this paper.
\subsection{Partitions and characters of $S_n$}
The reader will most likely be familiar with the following definitions, but we will recall them for the sake of clarity. Let $n$ be a positive integer. A partition of $n$ is a sequence $\lambda=(\lambda_1,\,\ldots,\,\lambda_t)$ of non-increasing positive integers such that $\sum_{i=1}^{t}\lambda_i=n$. We may sometimes write $\lambda_s$ with $s>t$, in this case the number is zero. We use $\mathcal{P}(n)$ to denote the set of all partitions of $n$. The \textit{Young diagram} (or \textit{diagram}) of a partition $\lambda$ is the set
\[
[\lambda]:=\{(i,j)\in\mathbb{N}\times\mathbb{N}\,|\hspace{0.2cm} 1\leq i \hspace{0.2cm} \text{and} \hspace{0.2cm} 1\leq j \leq \lambda_i\}.
\]
The elements of $[\lambda]$ are called \textit{cells} or \textit{boxes}, and the two indices are respectively the \textit{row index} and \textit{column index}. Given a partition $\lambda$, the symbol $\lambda'$ will denote the conjugate partition, which is to say the partition whose diagram is obtained by swapping the coordinates of $[\lambda]$. Note that $\lambda'_1$ coincides with the number of parts of $\lambda$. The \textit{$(i,j)$-hook} of a partition $\lambda$ is the subset of $[\lambda]$ defined as follows:
\[
H_{ij}(\lambda):=\{(s,t)\in[\lambda]\,|\hspace{0.2cm} s=i \hspace{0.2cm} \text{and}\hspace{0.2cm} t\geq j, \text{or} \hspace{0.2cm} t=j \hspace{0.2cm} \text{and}\hspace{0.2cm} s\geq i\}.
\]
A \textit{hook partition} is a partition $\lambda$ such that $[\lambda]=H_{11}(\lambda)$. The set of all hook partitions of $n$ will be written $\mathcal{H}(n)$.
We will also use the following nonstandard definition: we say that $\lambda$ is an $almost-hook$ partition if $[\lambda]\setminus H_{11}(\lambda)$ has just one element. We will use $\mathcal{AH}(n)$ to denote the set of all almost-hook partitions of $n$. \newline
\indent It is a well known fact (e.g. \cite{James}) that the simple $\mathbb{C}S_n$-modules up to isomorphism are the Specht modules indexed by partitions of $n$. For $\lambda\in\mathcal{P}(n)$ we will use $S^\lambda$ to denote the corresponding Specht module, and $\chi^\lambda$ the character corresponding to such module. We will recall the structure of these modules briefly, since it will be used in the next subsection. If $\lambda\in\mathcal{P}(n)$, a \textit{$\lambda$-tableau} is a bijection $[\lambda]\rightarrow\{1,\ldots,\,n\}$. $S_n$ acts naturally on the set of $\lambda$-tableaux. Let $t$ be a $\lambda$-tableau, $R_t$ and $C_t$ denote the subgroups of $S_n$ that stabilize the rows and columns of $t$ respectively. We define an equivalence relation by $t\sim s$ if $s=t\tau$ for some $\tau\in R_t$. The equivalence class of $t$ is denoted by $\{t\}$, and this is called a \textit{$\lambda$-tabloid}. $S_n$ acts naturally on the set of $\lambda$-tabloids. We call the permutation module of $S_n$ over this set $M^\lambda$. The Specht module $S^\lambda$ is the submodule of $M^\lambda$ spanned by the \textit{$\lambda$-polytabloids}: the elements of the form $\sum_{\sigma\in C_t}sign(\sigma)\{t\}\sigma$ where $t$ is a $\lambda$-tableau.\newline
\indent Another fundamental result in the representation theory of symmetric groups, which will be essential in almost all the proofs in this paper, is the Littlewood-Richardson Rule. To state it we will need some extra notation. Let $n>m$ be positive integers. Let $\lambda\in\mathcal{P}(n)$ and $\mu\in\mathcal{P}(m)$. We say that $\mu$ is a \textit{subpartition} of $\lambda$ if $[\mu]\subset[\lambda]$. If $\mu$ is a subpartition of $\lambda$ we define the following set:
\[
[\lambda\setminus\mu]:=[\lambda]\setminus[\mu].
\]
This is known as a \textit{skew diagram}. A sequence of positive integers $\mathcal{C}=(c_1,\ldots,c_n)$ is said to have \textit{weight} $\lambda$ if
\[
\lambda_k=|\{j:\hspace{0.2cm}c_j=k\}|
\]
for all $k\geq1$. An element $c_j$ in the sequence is said to be \textit{good} if $c_j=1$ or
\[
|\{i<j:\hspace{0.2cm}c_i=c_j\}|<|\{i<j:\hspace{0.2cm}c_i=c_j-1\}|.
\]
If every element in the sequence is good we say that the sequence is good.\newline
If $n>m$ are positive integers, we may identify $S_m\times S_{n-m}$ with any subgroup of $S_n$ obtained as the direct product of the symmetric group acting on a subset of $\{1,\ldots,\,n\}$ of cardinality $m$ and the symmetric group acting on its complement (these are called \textit{Young subgroups}, and from now on when we write $S_m\times S_{n-m}\leq S_n$ it will always be a group of this form). The Littlewood-Richardson Rule allows for the construction of the restrictions of characters of $S_n$ to $S_m\times S_{n-m}$ and inductions from these subgroups to $S_n$.
\begin{thm}[Littlewood-Richardson Rule]
Let $n>m$ be positive integers. Let $\mu\in\mathcal{P}(m)$ and $\nu\in\mathcal{P}(n-m)$. Consider $\chi^\mu\times\chi^\nu\in\mathrm{Irr}(S_m\times S_{n-m})$. Then we have:
\[
(\chi^\mu\times\chi^\nu)\uparrow^{S_n}=\sum_{\lambda\in\mathcal{P}(n)}c^\lambda_{\mu\nu}\chi^\lambda
\]
where $c^\lambda_{\mu\nu}$ is the number of ways of filling $[\lambda\setminus\mu]$ with positive numbers such that:\newline
\indent (i) the sequence obtained reading the numbers from right to left, top to bottom is a good 
\indent     sequence of weight $\nu$;\newline
\indent (ii) the numbers are strictly increasing along the columns from top to bottom;\newline
\indent (iii) the numbers are weakly increasing along the rows from left to right.

\end{thm}
For a proof see for example \cite[The Littlewood-Richardson rule]{James}. We refer to a filling of $[\lambda\setminus\mu]$ of the type above as a Littlewood-Richardson filling of $[\lambda\setminus\mu]$ of weight $\nu$. The above result shows that to prove that $[(\chi^\mu\times\chi^\nu)\uparrow^{S_n},\,\chi^\lambda]\neq0$, it suffices to find a Littlewood-Richardson (we will often abbreviate L-R) filling of $[\lambda\setminus\mu]$ of weight $\nu$. This will be used many times in this paper. Let us also highlight two simple but useful facts: by swapping $m$ and $n-m$ we see that if $c^\lambda_{\mu\nu}\neq0$ then $\nu$ is also a subpartition of $\lambda$; furthermore, in this case we have that $\lambda_1\leq \mu_1+\nu_1$ and $\lambda'_1\leq\mu'_1+\nu'_1$. And finally, let us state a simple to prove combinatorial remark, that will be used repeatedly in the sections below.
\begin{rem}\label{filling}
Let $n>m$ be positive integers. Let $\lambda\in\mathcal{P}(n)$ and $\mu\in\mathcal{P}(m)$, such that $\mu$ is a subpartition of $\lambda$; then the filling of $[\lambda\setminus\mu]$ given by the top to bottom numbering of the boxes of each column is a Littlewood-Richardson filling.
\end{rem}
In the following figure we give an example of this filling for the sake of clarity.
\[
\young(\times\times\times\times\times11,\times\times\times\times\times2,\times\times\times113,11,2)
\]
Here $\lambda=(7,\,6^2,\,2,\,1)$, $\mu=(5^2,\,3)$ and the weight of the filling is $(6,\,2,\,1)$.

\subsection{The Sylow Permutation Character}
Let $n$ be a positive integer, we define the following sets:
\[
\Omega_p(n):=\{\lambda\in\mathcal{P}(n)\, | \hspace{0.2cm} [\chi^\lambda,\,(1_{P_n})\uparrow^{S_n}]\neq0\}
\]
where $P_n$ is a Sylow $p$-subgroup of $S_n$. Clearly $\Omega_p(n)$ does not depend on the choice of the Sylow $p$-subgroup. The character $(1_{P_n})\uparrow^{S_n}$ is called a Sylow permutation character, and its significance and the reasons for trying to find its decomposition have been touched on in the introduction. As has been said there, the set $\Omega_p(n)$ has been determined in its entirety in \cite[Theorem~A]{GL1}, when $p$ is odd. In this paper we focus instead on the much more irregular $\Omega_2(n)$, which still elude a full description.\newline
\indent  We will now fix a particular Sylow $2$-subgroup of $S_n$ for each $n$. This will allow us to prove two interesting facts about $\Omega_2(n)$ before even reducing the problem using Clifford Theory. Both these results may already be known to experts, but we were not able to find proofs of them in the literature. For any $k\geq0$ we define $P_{2^k}$ as the Sylow $2$-subgroup of $S_{2^k}$ generated by the elements: $(1,2),(1,3)(2,4),\ldots,(1,2^{k-1}+1)\cdots(2^{k-1},2^k)$. Let $n>0$ be an integer, and $n=2^{a_1}+2^{a_2}+\cdots+2^{a_t}$ its binary expansion, with $a_1>\cdots>a_t$. We define $P_n$ as the Sylow $2$-subgroup of $S_n$ obtained as the product of the natural embeddings of the subgroups $P_{2^{a_k}}$ in $S_{\{m+1,\ldots,\,m+2^{a_k}\}}$ (the symmetric group acting on the set $\{m+1,\ldots,\,m+2^{a_k}\}$) where $m=2^{a_1}+\cdots+2^{a_{k-1}}$. With this notation fixed, we may observe that $(2k+1,2k+2)\in P_n$ for all $k\geq0$ such that $2k+2\leq n$, and that $P_n$ stabilizes the set $\{\{2k+1,\,2k+2\}:\hspace{0.2cm}2\leq2k+2\leq n\}$.\newline
\indent We will use the structure of the Specht modules to prove the following two results, so we will need to translate the character theoretic condition for $\lambda\in\Omega_2(n)$ into one that can be visualized in $S^\lambda$. This is done in the following lemma.
\begin{lem}\label{tab}
Let $n>0$ be an integer, and $\lambda\in\mathcal{P}(n)$. Then $\lambda\in\Omega_2(n)$ if and only if there exists a $\lambda$-tableau $t$ such that $\sum_{\pi\in P_n}e_t \pi\neq0$.
\end{lem}
\begin{proof}
The condition $[\,\chi^{\lambda}\downarrow_{P_n},\, 1_{P_n}]\,\neq0$ is equivalent to 
\[\exists v\in S^\lambda \setminus \{0\} \hspace{0.2cm} \forall \pi \in P_n \hspace{0.2cm} v\pi=v.
\] Since $\mathrm{char}(\mathbb{C})=0$, this in turn is equivalent to 
\[
\exists w\in S^\lambda\hspace{0.2cm} \sum_{\pi \in P_n}w\pi \neq0.
\] Since $S^\lambda$ is spanned by the $\lambda$-polytabloids this last condition is equivalent to asking that there exist a $\lambda$-tableau $t$ such that $\sum_{\pi \in P_n}e_t\pi \neq0$.
\end{proof}
If all the parts of the partition are even, then it is easy to find such a tableau. This is the first result, and it will be used in the proofs leading to the main theorem.
\begin{prop}\label{even}
Let $n>0$ be an integer and $\lambda\in\mathcal{P}(2n)$ a partition whose parts are all even, then $\lambda\in\Omega_2(2n)$.
\end{prop}
\begin{proof}
We will prove that the $\lambda$-tableau obtained by numbering the cells of $[\lambda]$ from left to right, top to bottom satisfies the condition in Lemma~\ref{tab}. Let $t$ be this $\lambda$-tableau.
We have
\[
\sum_{\pi \in P_n}e_t\pi=\sum_{\pi \in P_n}\sum_{\sigma\in C_t}sign(\sigma)\{t\}\sigma\pi.
\]
We will show that the coefficient of $\{t\}$ is positive. This coefficient is $\sum_{\pi,\sigma}sign(\sigma)$, where the sum is over all $\pi\in P_n$ and $\sigma\in C_t$ such that $\{t\}\sigma\pi=\{t\}$, or equivalently $\{t\}\sigma=\{t\}\pi^{-1}$. If $\sigma$ and $\pi$ satisfy this condition, then $sign(\sigma)=1$. This is because $P_n$ stabilizes the set $\{\{2k+1,\,2k+2\}:\hspace{0.2cm}0\leq k\leq n-1\}$, so for all such $k$, if $2k+1$ is in a certain row of $\{t\}\pi^{-1}$ then $2k+2$ is in that same row (because this condition is true for $\{t\}$). Any $\sigma\in C_t$ is the product of disjoint permutations acting on the individual columns. Since $\{t\}\sigma=\{t\}\pi^{-1}$, the previous condition tells us that the permutation acting on the column $2j+2$ has the same cyclic type as the one acting on the column $2j+1$, since each couple of adjacent elements in these columns of $\{t\}$ is of the type $\{2k+1,\,2k+2\}$. This means that $sign(\sigma)=1$. Since $1\in P_n\cap C_t$ the coefficient of $\{t\}$ is at least one.
\end{proof}
The next result we will prove is the result alluded to in the introduction as complementary to the main theorem. In this case we want to prove that certain partitions are not in $\Omega_2(n)$. This can be achieved with the following lemma.
\begin{lem}\label{temp}
Let $n>0$ be an integer. If $\lambda\in\mathcal{P}(n)$ is such that $(C_t\cap P_n)\setminus A_n\neq\emptyset$ for all $\lambda$-tableaux $t$, then $\lambda\notin\Omega_2(n)$.
\end{lem}
\begin{proof}
By Lemma~\ref{tab}, it is sufficient to prove that $\sum_{\pi\in P_n}e_t\pi=0$ for all $\lambda$-tableaux $t$. Let $t$ be a $\lambda$-tableau, and let $\tau\in C_t\cap P_n$ such that $sign(\tau)=-1$. Then we have:
\[
\sum_{\pi\in P_n}e_t\pi=\sum_{\pi\in P_n}\sum_{\sigma\in C_t}sign(\sigma)\{t\}\sigma\pi=\sum_{\pi\in P_n}\sum_{\sigma\in C_t}sign(\sigma)\{t\}\sigma(\tau\pi)=
\]
\[
=\sum_{\pi\in P_n}\sum_{\sigma\in C_t}sign(\sigma)\{t\}(\sigma\tau)\pi=-\sum_{\pi\in P_n}\sum_{\sigma\in C_t}sign(\sigma)\{t\}\sigma\pi=-\sum_{\pi\in P_n}e_t\pi.
\]
\end{proof}
\begin{prop}\label{big}
Let $n>0$ be an integer. If $\lambda\in\Omega_2(n)$ then $\lambda$ has at most $\lceil \frac{n}{2}\rceil$ parts.
\end{prop}
\begin{proof}
Let $\lambda\in\Omega_2(n)$. We know that $(2k+1,2k+2)\in P_n$ for all $k$ such that $2\leq 2k+2\leq n$. By Lemma~\ref{temp}, there must be a $\lambda$-tableau $t$ such that none of these $2$-cycles lies in $C_t$. In particular, if either of the elements $\{2k+1,2k+2\}$ lies in the first column of $t$, then the other must be in another. This means that the first column of $[\lambda]$ can have at most $\lceil\frac{n}{2}\rceil$ boxes.
\end{proof}
This is in fact an exact bound, since by Proposition~\ref{even} we know that $(2^n)=(2,\ldots,\,2)\in\Omega_2(2n)$, and Proposition~\ref{fund gen} will also yield $(2^n,\,1)\in\Omega_2(2n+1)$.\newline
\indent We will now prove two useful conditions for constructing $\Omega_2(n)$ inductively. These are essentially a slight modification to the arguments used in \cite{GL1} to reduce the odd prime case. The first proposition allows us to construct some of the partitions in $\Omega_2(2^{k+1})$ from partitions in $\Omega_2(2^k)$.
\begin{prop}\label{fund}
Let $k>0$ be an integer. Let $\mu,\nu\in\Omega_2(2^k)$ such that $\mu\neq\nu$. If $\lambda\in\mathcal{P}(2^{k+1})$ is such that $[\chi^\lambda,\,(\chi^\mu\times\chi^\nu)\uparrow^{S_{2^{k+1}}}]\neq0$, then $\lambda\in\Omega_2(2^{k+1})$.
\end{prop}
\begin{proof}
Suppose $\lambda\in\mathcal{P}(2^{k+1})$ and $\mu,\nu\in\Omega_2(2^k)$ are as stated above. We have a natural embedding $S_{2^k}\times S_{2^k}\trianglelefteq S_{2^k}\wr C_2\leq S_{2^{k+1}}$ (where the first is a Young subgroup and $C_2$ acts on it by swapping the elements the copies of $S_{2^k}$ act on). Then $P_{2^k}\times P_{2^k}\in\mathrm{Syl}_2(S_{2^k}\times S_{2^k})$ and $P_{2^k}\wr C_2\in\mathrm{Syl}_2(S_{2^{k+1}})$. Since $S_{2^k}\times S_{2^k}\trianglelefteq S_{2^k}\wr C_2$ we can apply Clifford Theory to $\chi^\mu\times\chi^\nu$. The inertial subgroup of this character is $S_{2^k}\times S_{2^k}$ (since $\mu\neq\nu$), so $(\chi^\mu\times\chi^\nu)\uparrow^{S_{2^k}\wr C_2}$ is irreducible. This means that $(\chi^\mu\times\chi^\nu)\uparrow^{S_{2^k}\wr C_2}$ is an irreducible constituent of $\chi^\lambda\downarrow_{S_{2^k}\wr C_2}$. Since $(S_{2^k}\times S_{2^k})\cap (P_{2^k}\wr C_2)=P_{2^k}\times P_{2^k}$ and $(S_{2^k}\times S_{2^k})\cdot (P_{2^k}\wr C_2)=S_{2^k}\wr C_2$, we have that:
\[
[(\chi^\mu\times\chi^\nu)\uparrow^{S_{2^k}\wr C_2}\downarrow_{P_{2^k}\wr C_2},\,1_{P_{2^k}\wr C_2}]=
[(\chi^\mu\times\chi^\nu)\downarrow_{P_{2^k}\times P_{2^k}}\uparrow^{P_{2^k}\wr C_2},\,1_{P_{2^k}\wr C_2}]=
\]
\[
=[(\chi^\mu\times\chi^\nu)\downarrow_{P_{2^k}\times P_{2^k}},\,1_{P_{2^k}\times P_{2^k}}]\neq0
\]
since $\mu,\nu\in\Omega_2(2^k)$.
\end{proof}
The second proposition allows us to construct the elements of $\Omega_2(n)$ starting from $\Omega_2(2^k)$.
\begin{prop}\label{fund gen}
Let $n>0$ be an integer. Write $n=\sum_{i=1}^t 2^{a_i}$ with $a_i$ distinct (here we do not require any specific ordering). Let $1\leq j<t$ and $m=\sum_{i=1}^j 2^{a_i}$. Let $\lambda\in\mathcal{P}(n)$. Then $\lambda\in\Omega_2(n)$ if and only if there exist $\mu\in\Omega_2(m)$ and $\nu\in\Omega_2(n-m)$ such that $[\chi^\lambda,\,(\chi^\mu\times\chi^\nu)\uparrow^{S_n}]\neq0$.
\end{prop}
\begin{proof}
Consider the following subgroups $P_m\times P_{n-m}\leq S_m\times S_{n-m}\leq S_n$. $P_m\times P_{n-m}$ is a Sylow $2$-subgroup of $S_n$. This means that $[\chi^\lambda\downarrow_{P_m\times P_{n-m}},\,1_{P_m\times P_{n-m}}]\neq0$ if and only if $\chi^\lambda\downarrow_{S_m\times S_{n-m}}$ has an irreducible constituent whose restriction to $P_m\times P_{n-m}$ satisfies the same condition. The irreducible characters of $S_m\times S_{n-m}$ satisfying this condition are exactly those of the form $\chi^\mu\times\chi^\nu$ where $\mu\in\Omega_2(m)$ and $\nu\in\Omega_2(n-m)$.
\end{proof}
Let's use this proposition to determine the sets $\Omega_2(n)\cap\mathcal{H}(n)$ and $\Omega_2(n)\cap\mathcal{AH}(n)$. In both cases we choose to parametrize these sets using the number of parts of the partitions.
\begin{lem}\label{H}
Let $n$ be a positive integer and $n=2^{a_1}+2^{a_2}+\cdots+2^{a_t}$ its binary expansion. We have that $(n-(l-1),\,1^{l-1})\in\Omega_2(n)\cap\mathcal{H}(n)$ if and only if $l\leq t$.
\end{lem}
\begin{proof}
We prove the result by induction on the number of elements in the binary expansion of $n$. If $n$ is a power of $2$ the result follows from \cite[Theorem~1.1]{G17}. Let $n=2^{a_1}+2^{a_2}+\cdots+2^{a_t}$ with $t\geq2$. Let $N=2^{a_1}+\cdots+2^{a_{t-1}}$. By Proposition~\ref{fund gen} and the Littlewood-Richardson Rule, $\lambda\in\Omega_2(n)$ if and only if there exist $\mu\in\Omega_2(N)$ and $\nu\in\Omega_2(2^{a_t})$ that are both subpartitions of $\lambda$ and such that $[\lambda\setminus\mu]$ has a L-R filling of weight $\nu$. If $\lambda\in\mathcal{H}(n)$ we must have $\mu\in\mathcal{H}(N)$ and $\nu\in\mathcal{H}(2^{a_t})$. Therefore, $\mu=(N-(l_1-1),\,1^{l_1-1})$ with $l_1\leq t-1$ and $\nu$ is the trivial partition. This implies that $\lambda$ has at most $t$ parts. Vice versa, if $\lambda=(n-(l-1),\,1^{l-1})$ with $2\leq l\leq t$ (the trivial partition is clearly in $\Omega_2(n)$), then we consider $\mu=(N-(l-2),\,1^{l-2})$ and the L-R filling consisting of filling each box with the number $1$.
\end{proof}
Let us observe here that, in the case $n=2^k$ we have $\Omega_2(2^k)\cap\mathcal{H}(2^k)=\{(2^k)\}$. This fact will be used in subsequent proofs without being explicitly recalled.
\begin{lem}\label{AH}
Let $n\geq4$ be a positive integer and $n=2^{a_1}+2^{a_2}+\cdots+2^{a_t}$ its binary expansion, with $a_1>\cdots> a_t$. We have that $(n-l,\,2,\,1^{l-2})\in\Omega_2(n)\cap\mathcal{AH}(n)$ if and only if $l\leq a_1+t-1$.
\end{lem}
\begin{proof}
We proceed by induction, as in the previous lemma. The base case is given by \cite[Theorem~1.3]{GuL}. Let $n=2^{a_1}+2^{a_2}+\cdots+2^{a_t}$ with $t\geq2$. Let $N=2^{a_1}+\cdots+2^{a_{t-1}}$. By Proposition~\ref{fund gen} and the L-R Rule, $\lambda\in\Omega_2(n)$ if and only if there exist $\mu\in\Omega_2(N)$ and $\nu\in\Omega_2(2^{a_t})$ that are both subpartitions of $\lambda$ and such that $[\lambda\setminus\mu]$ has a Littlewood-Richardson filling of weight $\nu$. If $\lambda\in\mathcal{AH}(n)$ we must have $\mu\in\mathcal{AH}(N)\cup\mathcal{H}(N)$ and $\nu\in\mathcal{AH}(2^{a_t})\cup\mathcal{H}(2^{a_t})$. Also, at least one of the two must be a hook partition: if $\mu\in\mathcal{AH}(N)$, then any L-R filling of $[\lambda\setminus\mu]$ has weight a hook partition. If $\nu\in\mathcal{H}(2^{a_t})\cap\Omega_2(2^{a_t})$, then it must be the trivial partition. At the same time, $\mu$ has at most $a_1+t-2$ parts, so $\lambda$ has at most $a_1+t-1$. If $\nu\in\mathcal{AH}(2^{a_t})\cap\Omega_2(2^{a_t})$ (which implies $a_t\geq 2$, otherwise there are no almost-hook partitions), then $\nu$ has at most $a_t$ parts. In this case $\mu\in\mathcal{H}(N)\cap\Omega_2(n)$, so $\mu$ has at most $t-1$ parts. Therefore, $\lambda$ has at most $a_t+t-1\leq a_1+t-1$ parts. Vice versa, if $\lambda=(n-l,\,2,\,1^{l-2})$ with $3\leq l\leq a_1+t-1$, then take $\mu=(N-(l-1),\,2,\,1^{l-3})$ and the Littlewood-Richardson filling of $[\lambda\setminus\mu]$ obtained by filling each box with $1$. The partition $(n-2,\,2)$ is in $\Omega_2(n)$, by Proposition~\ref{even} and the L-R Rule when $n$ is odd.
\end{proof}

This second result shows the maximality of the bound that we will prove in Theorem~\ref{main}.

\section{The power of 2 case}
To simplify notation we will define the following sets of partitions.
\begin{define}
Let $n$ and $k$ be positive integers. We let $\Delta(n,k)$ be the subset of $\mathcal{P}(n)\setminus\mathcal{H}(n)$ of partitions with at most $k$ parts.
\end{define}
We will proceed in the natural way suggested by the structure of the Sylow $2$-subgroups of $S_n$ and by Propositions \ref{fund} and \ref{fund gen}: in this section we will prove the desired result for powers of $2$, and then generalize to all positive integers in the next. For the first step we will use the two following lemmas.

\begin{lem}\label{1}
 Let $n\geq4$ be an integer; then
\[
\Delta(2^n,4)\subseteq\Omega_2(2^n).
\]
\end{lem}
\begin{proof}
We prove this lemma by induction. The base case $n=4$ has been checked using \cite{GAP}. Let $n>4$ and $\lambda\in\Delta(2^n,4)$. We will prove that $\lambda\in\Omega_2(2^n)$ by using Proposition~\ref{fund} and the Littlewood-Richardson Rule: it is sufficient to find a subpartition $\mu\in\Delta(2^{n-1},4)$ of $\lambda$ and a L-R filling of $[\lambda\setminus\mu]$ of weight $\nu\in\Delta(2^{n-1},4)\cup\{(2^{n-1})\}$, such that $\mu\neq\nu$.\newline
\indent Consider the following algorithm: mark the first $2^{n-1}$ boxes of $[\lambda]$ going from top to bottom, from left to right. The marked boxes form the Young diagram of a subpartition $\mu\in\mathcal{P}(2^{n-1})$ of $\lambda$. Furthermore, since $\lambda\notin\mathcal{H}(2^n)$, we have that $\mu\notin\mathcal{H}(2^{n-1})$ (box $(2,\,2)$ is necessarily marked); therefore $\mu\in\Delta(2^{n-1},4)$, as desired. We then fill $[\lambda\setminus\mu]$ with the L-R filling described in Remark~\ref{filling} and call it's weight $\nu\in\mathcal{P}(2^{n-1})$. If $\nu\in\Delta(2^{n-1},4)\cup\{(2^{n-1})\}$ and $\mu\neq\nu$, then we have obtained the desired result.\newline
\indent Let's suppose that $\mu=\nu$. This implies that $\nu'_1=\mu'_1=\lambda'_1$, which means that all columns of $[\mu]$, except possibly the last, must have $\lambda'_1$ boxes. If $\lambda'_1=1,\,2,\,4$ then, since $\lambda'_1$ divides $2^{n-1}$ and $\mu=\nu$, we have that:
\[
\lambda\in\{(2^n),\,(2^{n-1},\,2^{n-1}),\,(2^{n-2},\,2^{n-2},\,2^{n-2},\,2^{n-2})\}\subseteq\Omega_2(2^n)
\] (by Proposition~\ref{even}). Suppose $\lambda'_1=3$; if the last column of $[\mu]$ is made of two cells, then $[\nu]$ has a column with just one cell; if the last column of $[\mu]$ has one cell, then $[\nu]$ has a column with two cells. This means that if $\lambda'_1=3$ then $\mu\neq\nu$. \newline
\indent Suppose $\nu$ is a non-trivial hook partition. For this to occur, $[\lambda\setminus\mu]$ must have a single column with more than one box. Every column of $[\lambda\setminus\mu]$ from the second onward contains all the corresponding boxes in the respective column of $[\lambda]$, due to how the algorithm is defined. For this reason, the column of $[\lambda\setminus\mu]$ with more than one box must be the first or the second. \newline
\indent If the column is the first, we modify the output of the algorithm by marking the last unmarked box in this column, and unmarking the last box of the previous one. We maintain the same filling for all previously unmarked boxes, and input the number $2$ in the newly unmarked one. This is a L-R filling. This gives us two new partitions $\mu^0$ and $\nu^0$. Now $[\mu^0]$ has at least $5$ columns of which all but the last two have at least $2$ boxes. Meanwhile, $[\nu^0]$ has exactly two columns with more than one box. This means that both partitions are non hook partitions, and they are distinct. \newline
\indent If instead the column of $[\lambda\setminus\mu]$ is the second one, then we just replace the $1$ in the first column with a $2$, and conclude as before.
\end{proof}

The next lemma will allow us to increase the number of parts as we increase $n$.
\begin{lem}\label{2}
Let $n\geq k\geq 4$ be integers such that $\Delta(2^n,k)\subseteq\Omega_2(2^n)$ and $\Delta(2^{n+1},k)\subseteq\Omega_2(2^{n+1})$; then we have that
\[
\Delta(2^{n+1},k+1)\subseteq\Omega_2(2^{n+1}).
\]
\end{lem}
\begin{proof}
Suppose that $n$ and $k$ satisfy the above conditions. Let $\lambda\in\Delta(2^{n+1},k+1)\setminus\Delta(2^{n+1},k)$. We will prove that $\lambda\in\Omega_2(2^{n+1})$, as in Lemma~\ref{1}, by using Proposition~\ref{fund} and the Littlewood-Richardson Rule. We will construct two partitions $\mu,\nu\in\Delta(2^n,k)\cup\{(2^n)\}\subseteq\Omega_2(2^n)$ such that $\mu\neq\nu$ and there is a L-R filling of $[\lambda\setminus\mu]$ of weight $\nu$. Consider the following algorithm on $[\lambda]$: mark the first $k$ boxes of the first column, the first two of the second, and then continue marking from left to right, top to bottom until the marked cells form the Young diagram of a partition $\mu\in\Delta(2^n,k)$. Note that $\mu$ cannot have $k+1$ parts, because this would require marking more than $2^n$ boxes (we will omit similar observations in other subsequent proofs). Then fill the boxes of $[\lambda\setminus\mu]$ as in Remark~\ref{filling} and let $\nu\in\mathcal{P}(2^n)$ be its weight. \newline \indent Firstly, let us prove that $\nu$ has at most $k$ parts. To do this, we will prove this slightly stronger fact: in applying the algorithm, we mark at least one cell in every column of $[\lambda]$ with at least $k$ cells. Suppose that $\lambda$ is such that this doesn't happen. Let $t$ be the index of the first untouched column (which must have at least $k$ elements). Since $k+2\leq n+2<2^n$, $t$ must be at least $4$. The number of marked boxes is $k+2+t-3=2^n$. The number of unmarked boxes in the first $t$ columns is at least $1+k-2+(k-1)(t-3)+k$, and it must not exceed $2^n$ (the total number of unmarked boxes). This gives the following inequality:
\[
1+k-2+(k-1)(t-3)+k\leq k+2+t-3
\]
\[
3\leq(k-2)(t-3)+k-3\leq0.
\]
\indent The reason why we proved this stronger result is that this can also be used to prove that $\mu\neq\nu$. Suppose that $\mu$ and $\nu$ have the same number of parts. By the previous observation there must be at least one column of $[\lambda]$ with $k+1$ boxes of which only one is marked. Due to how the algorithm is defined, this column can't be the first nor the second. Therefore the second column of $[\lambda]$ is made of $k+1$ boxes, of which only two are marked (the third can't be marked, otherwise all boxes in the second row would also be marked). This means that $[\mu]$ has one column with $k$ boxes, and all others with two or less. Meanwhile, $\nu$ has a column with $(k+1)-2=k-1\geq3$ boxes, and therefore $\mu\neq\nu$. This leaves the case in which $\nu$ is a non-trivial hook partition. \newline \indent Suppose that $\nu$ is a non-trivial hook partition. By the previous observation, $\nu$ cannot have $k$ parts. This means that if we modify the partitions $\mu$ and $\nu$ without changing the number of their parts, then they will still be distinct. Just as in the proof of Lemma~\ref{1}, if $\nu$ is a hook partition, then $[\lambda\setminus\mu]$ has only one column with more than one box. If the other boxes of $[\lambda\setminus\mu]$ are not all in the same row, it suffices to change the content of the last box of the last of these rows (excluding all elements of the column with more than one box) from a $1$ to a $2$. This changes $\nu$ into a non hook partition, without modifiying the number of parts.\newline
\indent The last remaining case is that in which all the cells of $[\lambda\setminus\mu]$ lie in the union of a column and a row. This is impossible. If this were to happen, this would mean that the row considered is the $(k+1)$-th, since the last box of the first column of $[\lambda]$ is left unmarked. Therefore, the number of cells in $[\lambda\setminus\mu]$ is at most $\lambda_{k+1}+k-1$, while the number of marked cells is at least $k(\lambda_{k+1}-1)$; this results in the inequality $\lambda_{k+1}+k-1\geq k(\lambda_{k+1}-1)$, which isn't satisfiable for $k\geq4$ and $\lambda_{k+1}\geq 2^n-(k-1)\geq13$.
\end{proof}
We will now give an example of the algorithm described in Lemma~\ref{2}.
\begin{example}
Let $\lambda=(25,\,2^3,\,1)\in\Delta(32,5)\setminus\Delta(32,4)$. The algorithm gives us the following figure
\[
\young(\times\times\times\times\times\times\times\times\times\times\times\times1111111111111,\times\times,\times1,\times2,1)
\]
The output is $\mu=(12,\,2,\,1^2)$ e $\nu=(15,\,1)$. We have $\nu\in\mathcal{H}(16)\setminus\{(16)\}$ so we must apply the modification described in Lemma~\ref{2}.
\[
\young(\times\times\times\times\times\times\times\times\times\times\times\times1111111111111,\times\times,\times1,\times2,2)
\]
Now $\mu=(12,\,2,\,1^2)$ e $\nu=(14,\,2)$.
\end{example}
The proof of the desired bound in the case of powers of $2$ is now essentially complete.
\begin{prop}
Let $n\geq4$ be an integer. Then
\[
\Delta (2^n,n)\subseteq\Omega_2(2^n).
\]
\end{prop}
\begin{proof}
We will prove by induction on $n\geq4$ that for all $m\geq n$, we have $\Delta(2^m,n)\subseteq\Omega_2(2^m)$, which is equivalent to the statement above. The base case is given by Lemma~\ref{1}. Let $n>4$ such that for all $N\geq n-1$, we have $\Delta(2^N,n-1)\subseteq\Omega_2(2^N)$. Let $m\geq n$, then $m,\,m-1\geq n-1\geq4$, so $\Delta(2^{m-1},n-1)\subseteq\Omega_2(2^{m-1})$ and $\Delta(2^{m},n-1)\subseteq\Omega_2(2^{m})$. By Lemma~\ref{2} we have $\Delta(2^m,n)\subseteq\Omega_2(2^m)$.
\end{proof}

To prove the corresponding result for all positive integers we will require a slightly stronger bound for powers of $2$, obtained by excluding almost-hook partitions. This we present as the following proposition.

\begin{prop}\label{extension}
If $n\geq4$ is an integer, then
\[
\Delta(2^n, 2n-3)\setminus\mathcal{AH}(2^n)\subseteq\Omega_2(2^n).
\]
\end{prop}
\begin{proof}
We prove this result by induction on $n\geq4$. The base case has been checked with \cite{GAP}. Let $n>4$, such that the statement is true for $n-1$. Let $\lambda\in\Delta(2^n, 2n-3)\setminus(\Delta(2^n,n)\cup\mathcal{AH}(2^n))$. Let $1\leq k\leq n-3$ be the integer such that $n+k$ is the number of parts of $\lambda$. Just as in Lemma~\ref{2}, we will provide an algorithm that gives two partitions $\mu,\nu\in\Omega_2(2^{n-1})$, such that $\mu\neq\nu$ and there is a L-R filling of $[\lambda\setminus\mu]$ of weight $\nu$. The algorithm is the following: mark the first $n-1$ cells in the first column of $[\lambda]$, then the first $\mathrm{max}\{2,\,\lambda_2'-(n-2)\}$ in the second, then the first $\mathrm{max}\{1,\,\lambda_i'-(n-1)\}$ in the $i$-th column (starting from $i=3$), until $2^{n-1}$ cells have been marked or the last column has been reached. If the number of marked cells is less than $2^{n-1}$, we reach that number by marking boxes from left to right, top to bottom (we will refer to this phase as the \textit{filling phase}). As before, we call the partition corresponding to the obtained diagram $\mu$, while $\nu$ is the partition obtained as the weight of the usual filling of $[\lambda\setminus\mu]$ (Remark~\ref{filling}). Clearly $\mu\in\Delta(2^{n-1},n-1)$.\newline
\indent Firstly, we will prove that $\nu$ has at most $n-1$ parts. To do this we will prove a stronger fact which will be used later: the algorithm can't terminate before marking all the required cells in each of the columns with at least $n-1$ boxes, plus one more. If the algorithm were to terminate by marking $\mathrm{max}\{1,\,\lambda_i'-(n-1)\}$ or less boxes in a column with $n-1$ boxes or more, this would mean that, since each column has at most $2n-3$ boxes, the number of marked cells in each column is at most $n-1$. Let $l$ be the number of columns of $[\lambda]$ with at least $n-1$ cells. The number of marked cells must be $2^{n-1}$, therefore we must have:
\[
l\geq \frac{2^{n-1}}{n-1}.
\]
Note that, since the \textit{filling phase} isn't necessary in this case, in each column with at least $n-1$ elements from the third onward the number of marked cells is strictly smaller than the number of unmarked cells. Let's call the number of marked cells in these columns $t$: the number of unmarked cells in these columns is at least $t+l-2$. Let $m$ be the number of marked cells in the first $l$ columns and $u$ the number of unmarked cells in these same columns:
\[
(n-1)+(k+2)+t\geq m=2^{n-1}\geq u\geq (k+1)+(n-3)+t+l-2
\]
\[
5\geq l \geq \frac{2^{n-1}}{n-1}.
\]
This is possible only when $n=5$. If $n=5$, there are only five partitions of $16$ whose diagrams have at most $4$ rows and $5$ columns: the conjugate partitions of $(4^4)$, $(4^3,\,3,\,1)$, $(4^3,\,2^2)$, $(4^2,\, 3^2,\,2)$ e $(4,\,3^4)$. In this case $2n-3=7$ so $[\lambda]$ has at most $7$ cells in each column. In the situation above, the algorithm marks at most $3$ cells in each column after the second: therefore, $\mu'$ can't be any of the first three. If $\mu'$ were one of the other two partitions, then $\mu_3'=3$, which means that $\lambda_3'=7$, so $\lambda_2'=7$ and $\mu_2'=4$. This leaves $(4^2,\,3^2,\,2)$. In this case $\lambda'=(7^4,\,6,\,\ldots)$, which isn't a partition of $32$. This proves that $\nu$ has at most $n-1$ parts. \newline
\indent Suppose $\mu=\nu$. This means that there is a column in $[\lambda]$ that comes after the second and has $n+i$ boxes, with $i\geq0$, of which exactly $i+1$ have been marked. We deduce from this that no boxes were marked in the $(i+3)$-th column during the \textit{filling phase} of the algorithm. The second column is $\lambda_2'=n+j$ with $j\geq i$. Since $j+2\geq i+2$, exactly $j+2$ boxes have been marked in the second column. Therefore, $[\mu]$ has one column with $n-1$ boxes, and all others with at most $j+2$; meanwhile, $\nu$ has at least one column with $n-1$ boxes and at least one with $n-2$. If $n-2>j+2$, then $\mu\neq\nu$. This leaves $j=n-4,n-3$. In these cases, unmark the last cell marked and mark the first unmarked cell in the first column of $[\lambda]$. We replace $\mu$ with the partition corresponding to the new diagram made of marked boxes. We then fill $[\lambda\setminus\mu]$ using the same rule as before. The new partitions $\mu$ and $\nu$ are now such that $\mu$ has $n-1+1\leq 2(n-1)-3$ parts, and a column with $j+2\geq n-2\geq 3$ boxes, so $\mu\in\Omega_2(2^{n-1})$ since it isn't a hook or almost-hook partition. Meanwhile, $\nu$ has $n-1$ parts and a column with $n-2\geq 3$ boxes, so it isn't a hook partition. \newline
\indent The remaining case is that in which $\nu$ is a non trivial hook partition. By the proof above, we know that, in this case, $\nu$ has at most $n-2$ parts. Just as in Lemma~\ref{2} we can suppose that the unmarked cells lie in the union of a column and a row of $[\lambda]$ (otherwise we can change a $1$ to a $2$ in the filling). The column must be the first: we know that the boxes in the $n$-th and $n+1$-th positions of this column are left unmarked. 
The row must be the first. If this weren't true, let $l$ be the number of unmarked elements in this row that aren't in the first column. Counting marked and unmarked boxes, we have $l+(n-2)\geq l+(n-1)$. Therefore the algorithm has terminated before the \textit{filling phase}, marking all boxes of the columns from the second onward until it stops before a column with just one part. This can only happen if $\lambda\in\mathcal{H}(2^n)\cup\mathcal{AH}(2^n)$.

\end{proof}

\section{The general case}
The first step towards generalizing the bound we have now proven for powers of two will be to give a generalization of Proposition~\ref{extension}. This will yield most cases and, at the same time, give a useful instrument for handling the rest.

\begin{prop}\label{help}
Let $n\geq16$ be an integer. Suppose $k$ is the greatest integer such that $2^k\leq n$. Then 
\[
\Delta(n,2k-3)\setminus\mathcal{AH}(n)\subseteq\Omega_2(n).
\]
\end{prop}
\begin{proof}
Suppose $n$ and $k$ satisfy the conditions above, and $\lambda\in\Delta(n,2k-3)\setminus\mathcal{AH}(n)$. We can assume that $n$ is even: if $n$ is odd, just remove any box of $[\lambda]$ such that the remaining diagram doesn't correspond to a hook or almost-hook partition. We will use the Littlewood-Richardson Rule and Proposition~\ref{fund gen} to reduce to the base case proven in Proposition~\ref{extension}. We will do so by finding $\mu\in\Delta(2^k,2k-3)\setminus\mathcal{AH}(2^k)$ a subpartition of $\lambda$, and a L-R filling of $[\lambda\setminus\mu]$ of weight $\nu$, a partition with all parts even (Proposition~\ref{even}).\newline
\indent We define the following algorithm: consider the last box in each column of $[\lambda]$, going from right to left, mark the maximum even number of these, such that the remaining boxes form the diagram of a partition which is not a hook nor an almost-hook partition. Remove the marked boxes of $[\lambda]$ and the columns in which no boxes have been marked. Repeat on the new diagram. Interrupt the algorithm once exactly $n-2^k$ boxes have been marked or the diagram that's left is empty. Let $\mu$ be the subpartition of $\lambda$ corresponding to the diagram made up of all the unmarked cells in $[\lambda]$. Then use the usual L-R filling (Remark~\ref{filling}) of $[\lambda\setminus\mu]$: it's weight is clearly a partition with all parts even ($\nu_1$ is the number of cells marked at the first iteration, $\nu_2$ the number marked at the second, and so on). All that's left to prove is that it is, in fact, a partition of $n-2^k$. To do this, we will show that the algorithm cannot terminate before marking that many boxes. Suppose that the algorithm ends without marking $n-2^k$ boxes (so it reaches an empty set of boxes). Due to the algorithm's definition there can't be more than $2k-3$ unmarked boxes in the first column, $2k-4$ in the second, $2k-5$ in the third, and so on. This is clear when the algorithm operates \textit{smoothly}, by which we mean that we can always mark the maximal number of boxes without leaving a hook or almost-hook partition. If at some iteration the algorithm cannot remove a box from the second column because this would make the remaining diagram an almost-hook partition, this means that there are only three boxes left in the second column and at most two in the ones following it. One box in the fourth column is marked at this iteration. The algorithm resumes its function regularly from the third or fourth column, which means that there are at most $3$ unmarked boxes in the second column, $2$ in the third, and $1$ in the fourth. If the algorithm cannot mark a cell in the third column, this means that there are only two unmarked cells in the second. Just as before we conclude that the number of unmarked boxes is at most $2$ in the second, third, and fourth column, and $1$ in the fifth. Therefore, the number of unmarked boxes can be at most $\frac{(2k-3)(2k-2)}{2}<2^k$, so the number of marked boxes must be greater than $n-2^k$ which gives a contradiction.
\end{proof}
The following example illustrates the functioning of the algorithm described in Proposition~\ref{help}.
\begin{example}
Let $n=23=16+4+2+1$ and $\lambda=(5^3,\,4^2)$. Then we can reduce to $n=22$ by taking for example $(5^3,\,4,\,3)$. We then apply the algorithm and obtain the following figure.
\[
\young(~~~~~,~~~~1,~~~12,~~~2,~11)
\]
In this case $\mu=(5,\,4,\,3^2,\,1)$ and $\nu=(4,\,2)$.
\end{example}

For many $n$, Proposition~\ref{help} gives a stronger result than the one we have set out to prove. In particular, let $n=\sum_{i=1}^{t}2^{a_i}$, with $a_1>\cdots>a_t$, if $a_1+t-1\leq 2a_1-3$, then Proposition~\ref{help} and Lemma~\ref{AH} imply the desired result. There are, however, still an infinite number of cases that aren't included in the above result ($t\in\{a_1-1,\,a_1,\,a_1+1\}$). These will be eliminated in the next and final proof of this paper.

\begin{thm}\label{main}
Let $n$ be a positive integer. Let $n=\sum_{i=1}^{t}2^{a_i}$, with $a_1>\cdots>a_t$. Let $k_n$ be the maximum integer such that $\Delta(n,k_n)\subseteq\Omega_2(n)$. Then:
\[
k_n=\begin{cases}
\infty & \mbox{if $n\leq 5$ (formally, since $\mathcal{P}(n)\setminus\mathcal{H}(n)\subseteq\Omega_2(n)$)}\\
1 & \mbox{if $n=6, 8, 10$}\\
2 & \mbox{if $n=9$}\\
a_1+t-1 & \mbox{if $n=7$ or $n\geq11$}
\end{cases}
\]
\end{thm}
\begin{proof}
 All cases up to $n=31$ have been verified using \cite{GAP}. Let $n\geq32$ and $n=\sum_{i=1}^{t}2^{a_i}$ its binary decomposition, with $a_1>\cdots>a_t$. Firstly, let us remind the reader that Lemma~\ref{AH} provides the description of $\Omega_2(n)\cap\mathcal{AH}(n)$, and at the same time proves $k_n\leq a_1+t-1$, so we just have to prove that $\Delta(n,a_1+t-1)\setminus\mathcal{AH}(n)\subseteq\Omega_2(n)$. To do this we will use Proposition~\ref{fund gen} and Proposition~\ref{help}. Let $\lambda\in\Delta(n,a_1+t-1)\setminus\mathcal{AH}(n)$. By Proposition~\ref{help} we may assume that $a_1+t-1\geq 2a_1-2$ (which is to say $t=a_1-1,\,a_1,\,a_1+1$) and that $\lambda$ has at least $2a_1-2$ parts. We will proceed by describing a similar algorithm on $[\lambda]$ to the one described in Proposition~\ref{help}, such that the returned partitions $\mu\in\mathcal{P}(2^{a_1})$ and $\nu\in\mathcal{P}(n-2^{a_1})$ are of the kinds whose behavior has been characterised in the various lemmas and propositions above. In particular we want to be able to apply Proposition~\ref{help} to $\mu$ and $\nu$, so for this part of the proof we will assume $n-2^{a_1}\geq16$ (equivalently $a_2\geq4$).\newline
\indent The algorithm is the following: mark all boxes in the rows numbered $2a_1-2$ onward; if the number of marked boxes hasn't exceeded $n-2^{a_1}$, mark the last box of each column from right to left (if one of such boxes has already been marked, leave as is), then the second last, and so on until $n-2^{a_1}$ cells have been marked. During the process, make sure to skip marking a box if the resulting diagram of unmarked boxes does not correspond to a partition or it corresponds to a hook or almost-hook partition. As before, we will call $\mu$ the partition corresponding to the diagram made up of unmarked boxes, and $\nu$ the partition obtained as the weight of the usual L-R filling of $[\lambda\setminus\mu]$ (Remark~\ref{filling}). The first thing we must prove is that the algorithm always marks exactly $n-2^{a_1}$ cells. The only way this may not occur is if $\lambda_{2a_1-2}+\lambda_{2a_1-1}+\lambda_{2a_1}>n-2^{a_1}$. If $a_1\geq 6$ some simple approximations show that this is impossible. To simplify notation, let $s:=t-a_1+2$ (the maximum number of parts following $\lambda_{2a_1-3}$):
\[
\lambda_{2a_1-2}+\lambda_{2a_1-1}+\lambda_{2a_1}\leq \frac{n}{2a_1-2}\cdot s< \frac{2^{a_1+1}}{2a_1-2}\cdot s
\]
\[
n-2^{a_1}\geq 1+\cdots+2^{a_1+s-4}=2^{a_1+s-3}-1
\]
\[
1\geq \frac{8}{2a_1-2}+\frac{1}{2^{a_1-2}}\geq \frac{2^{4-s}}{2a_1-2}\cdot s + \frac{1}{2^{a_1+s-3}}.
\]
The inequality for the remaining cases (those with $a_1=5$ and $a_2=4$) has been checked with \cite{GAP}. This means that $\mu$ and $\nu$ are partitions of the correct numbers. By the definition of the algorithm $\mu\in\Delta(2^{a_1},2a_1-3)\setminus\mathcal{AH}(2^{a_1})$. \newline
\indent Now we will prove that $\nu$ has at most $2a_2-3$ parts. If this weren't true, we would have had to mark more than the set comprised of the last $2a_2-3$ boxes in each column (except for the boxes $(1,2)$, $(2,2)$, $(3,2)$ or $(1,2)$, $(2,2)$, $(1,3)$, $(2,3)$). The number of such boxes is at least $\lambda_1+\cdots +\lambda_{2a_2-3}-4$. Therefore it suffices to prove that $\lambda_1+\cdots +\lambda_{2a_2-3}-4\geq n-2^{a_1}$ or equivalently that $\lambda_{2a_2-2}+\cdots+\lambda_{a_1+t-1}+4\leq 2^{a_1}$. As before, this is easy to prove for \textit{large} $a_1$ (in this case $a_1\geq9$), meanwhile, cases $a_1=6,\,7,\,8$ have been checked by hand using better approximations, and the case $a_1=5$ has been checked using \cite{GAP}. For $a_1\geq 9$ the simple approximations used are the following (again $s:=t-a_1+2$):
\[
a_1+t-1-2a_2+3\leq 2a_1+s-3-2(a_1-4+s)+3= 8-s
\]
\[
\lambda_{2a_2-2}+\cdots+\lambda_{a_1+t-1}+4\leq \lceil\frac{n}{a_1+t-1}\rceil\cdot (8-s)+4<\frac{2^{a_1+1}}{2a_1-2}\cdot 7+11
\]
\[
\frac{14}{2a_1-2}+\frac{11}{2^{a_1}}\leq 1.
\]
The only cases remaining now are those in which $\nu\in\mathcal{H}(n-2^{a_1})\cup\mathcal{AH}(n-2^{a_1})$. In this case we will prove that the number of parts of $\nu$ is at most $3\leq a_1-2\leq t-1$ so $\nu\in\Omega_2(n-2^{a_1})$ (by Lemma~\ref{H} and Lemma~\ref{AH}). Suppose that $\nu$ is a hook or almost-hook partition with at least four parts. This means that the algorithm has continued to the second phase, and has marked the last three cells in each column (excluding the ones of the type mentioned above). Since $[\nu]$ has at most two columns with more than one cell, only one of which (the first) has more than two cells, the first column of $[\nu]$ must correspond to the marked cells in the first column of $[\lambda]$, in which at least three boxes have been marked. At the third pass of the second phase of the algorithm, there must have been only one available cell to mark, and this must have been in the first column; this means that removing any other cell from the diagram gives an almost-hook partition or something that isn't a partition. This means that $\mu$ is of the form $(m,\,2^2)'$ or $(m,\,3)'$ and in both cases $m\leq2a_1-4$. However $m+4\leq2a_1<2^{a_1}$ gives a contradiction.\newline
\indent There are now only a finite number of remaining cases: those $n>32$ such that $t\geq a_1-1$ and at the same time $a_2<4$. These numbers are $39,\,43,\,45,\,46,\,47,\,79$. Each case can be solved by hand by reducing to a smaller $n$. For example, let $\lambda\in\mathcal{P}(39)$ be a partition with $8$ parts. Since $4<\frac{39}{8}<5$, this means that $\lambda_1\geq5$ and $\lambda_8\leq4$. If $\lambda_8=1,2,4$, let $\mu\in\mathcal{P}(38),\mathcal{P}(37),\mathcal{P}(35)$ (respectively) be the partition obtained by removing the last part, if $\lambda_8=3$, let $\mu\in\mathcal{P}(35)$ be the partition corresponding to the diagram obtained from $[\lambda]$ by removing the last row and the last cell of the last column. In every case $\mu$ is neither a hook nor almost-hook partition and the usual Littlewood-Richardson filling has weight corresponding to the trivial partition (of $1,2$ or $4$). We will omit the proof of the other cases, since they are very similar to the one above.
\end{proof}


\end{document}